\documentclass[preprint]{elsarticle}
\usepackage{amsbsy}
\usepackage{amsfonts}
\usepackage{amsmath}

\newdefinition{defin}{Definition}
\newtheorem{thm}{Theorem}
\newtheorem{prop}[thm]{Proposition}
\newproof{proof}{Proof}

\newtheorem{lem}{Lemma}

\newcommand{\no}{{\mathbb{N}}}

\begin{document}

\begin{frontmatter}
\title{The concept of duality for automata over a changing alphabet and generation of a free group by such automata}
\author[rvt]{Adam Woryna}
\ead{adam.woryna@polsl.pl}
\address[rvt]{Silesian University of Technology, Institute of Mathematics, ul. Kaszubska 23, 44-100 Gliwice, Poland}

\begin{abstract}
In the paper, we deal with the notion of  an automaton  over a changing alphabet, which generalizes the concept of a Mealy-type automaton. We modify the methods based on the idea of a dual  automaton and its action used by B. Steinberg et al. (2011) and M. Vorobets and Ya. Vorobets (2007,2010) \cite{11,13} and adapt them to automata over a changing alphabet. We show that this modification provides some naturally defined automaton representations of a free nonabelian group by a 2-state automaton over a changing alphabet.
\end{abstract}

\end{frontmatter}

\section{Introduction}

The combinatorial language of transducers (also known as Mealy automata) turned out to be useful in the group theory for describing groups acting on a rooted tree.  By definition, a Mealy automaton $A$ over a finite alphabet $X$ is a set $Q$ of states equipped with a transition function $\varphi\colon Q\times X\to Q$ and an output function $\psi\colon Q\times X\to X$.  Every state $q\in Q$, through a transition and an output function, induces a transformation  of a tree of finite words $X^*$ over the alphabet $X$. If these transformations are all invertible, then $A$ is called invertible and the group generated by them (with the composition of mappings as a product) is called the group generated by $A$. The extensive presentation of the theory of groups generated by Mealy automata is included in the survey paper \cite{6} (for various aspects and constructions regarding automaton groups and semigroups see also \cite{22,18,19,21,20}).

In recent years the automaton representations of free nonabelian groups are intensively investigated (see for example \cite{1,4,9,11,12,13,15}). The combined results of \cite{11}, \cite{12} and \cite{13} are based on the notion of a dual automaton of a given Mealy automaton and describe the action of a dual automaton to prove that the constructed automata generate free nonabelian groups (see also \cite{10} generalizing some results of \cite{11}). In this way, for every $n\geq 3$ a free nonabelian group of rank $n$ was presented as a group generated by an $n$-state Mealy automaton over the  binary alphabet $X=\{0,1\}$. The invertible automata considered in \cite{11,12,13} have the property that their dual automata are also invertible. Such automata are called bireversible. The groups generated by bireversible automata are very convenient  to study via the action of the corresponding dual automata. Moreover, the transformations of $X^*$ defined by all bireversible automata over the alphabet $X$ form a group  isomorphic to the group of the so-called geometric automorphisms of a free group with the basis $X$ (see \cite{4} or \cite{7}).

Every  Mealy automaton has the same structure in all the moments of the discrete time scale and it defines the group acting on a regular rooted tree. However,   finding an automaton representation of a given abstract group $G$ is difficult in general. For example, the problem of  constructing a 2-state Mealy automaton (over some finite alphabet $X$) generating a free nonabelian group is still open and we even do not know if such an automaton exists. Thus it would be interesting to investigate various directions in generalization of the concept of a Mealy automaton, which allow to find and study representations of   groups. This paper constitutes such an approach. Namely, we use the idea of a changing alphabet and an automaton over a changing alphabet to provide some naturally defined automaton representations of a free nonabelian group by a 2-state automaton over a changing alphabet.

Automata over a changing alphabet  constitute a particular subclass of the class of time-varying automata   working over a changing alphabet (the notion of a time-varying automaton previously concerned automata over a fixed alphabet, and the research on them was mainly concentrated  on their structural properties - see \cite{8}).
The class of time-varying automata over a changing alphabet was first defined in the paper by the author \cite{14}, where he   showed how  this concept can be used for describing groups acting on  spherically homogeneous rooted trees which are not necessarily regular (see also \cite{15,16,17}).  Groups acting on regular rooted trees are studied via the language of time-varying automata (over a fixed alphabet) also in the papers by Anna Erschler \cite{2,3}, and the first implicit appearance of these groups dates back to the paper by R. Grigorchuk \cite{5}.
In the paper \cite{15} we constructed other realizations of a free nonabelian group of rank 2 as a group generated by  time-varying automata, one of which was given by a 2-state automaton over a changing alphabet. However,
transition functions of that automaton are of  diagonal type, which implies quite complicated description of output functions and in this way misses the essential  feature of desirable automaton presentations, namely, simplicity.

A changing alphabet we define as an infinite sequence
$$
X=(X_1, X_2, \ldots)
$$
of finite sets, which are called sets of letters. The set $X^*_{(1)}$ of all words over the changing alphabet $X$ we define as a disjoint union of cartesian products
\begin{equation}\label{e-1}
X_1\times X_2\times\ldots \times X_m,\;\;\;m\geq 1.
\end{equation}
It is assumed that $X^*_{(1)}$ contains also the empty word denoted by $\emptyset$. Since the elements of a word $w\in X^*_{(1)}$ are called letters, we will not separate them by commas and we will write $w=x_1x_2\ldots x_m$, where $x_j\in X_j$  are the letters. For a given $m\geq 1$ the product (\ref{e-1}) constitutes the set of words of the length $m$.  The set $X^*_{(1)}$  has the structure of a spherically homogeneous rooted tree.  The empty word is the root of this tree and two words are connected if and only if they are of the form $w$ and $wx$ for some word $w\in X^*_{(1)}$ and a letter  $x\in X_{|w|+1}$, where $|w|$ denotes the length of  $w$. The $m$-th level ($m\geq 1$)  of the tree $X^*_{(1)}$ (that is the set of vertices at the distance $m$ from the root) consists of all words of the length $m$, and thus coincides with the product (\ref{e-1}). It is worth to see that the tree $X^*_{(1)}$ is regular if and only if the sizes of the sets of letters coincide.

We define an automaton $A$ over a changing alphabet $X=(X_1, X_2, \ldots)$ as a set $Q$ of states equipped with a sequence $\varphi=(\varphi_1, \varphi_2, \ldots)$ of transition functions $\varphi_i\colon Q\times X_i\to Q$ and a sequence $\psi=(\psi_1, \psi_2, \ldots)$ of output functions $\psi_i\colon Q\times X_i\to X_i$. If  $Q$ is finite, then the automaton $A$ is called finite. The sequences of transition and output functions define for every state $q\in Q$ a transformation $A_{1, q}\colon X^*_{(1)}\to X^*_{(1)}$, which we call an automaton function (corresponding to the state $q$ in the first transition of the automaton $A$). Every automaton function  preserves the lengths of words and common beginnings. This implies that the set of  automaton functions defined by  automata   over a changing alphabet $X$ coincides with the set of all endomorphisms of the tree $X^*_{(1)}$. Moreover, if an automaton function is invertible, then the inverse transformation is also an automaton function. Thus the set of all invertible automaton functions coincide with the group $Aut(X^*_{(1)})$ of automorphisms of the tree $X^*_{(1)}$. In this group we distinguish the proper subgroup $FA(X)$, which consists of all invertible automaton functions defined by finite automata.
For a given invertible automaton $A=(X, Q, \varphi, \psi)$  the automaton functions $A_{1, q}$ defined by the states $q\in Q$ of $A$  are all invertible and  the group
$$
G(A)=gp\{A_{1, q}\colon q\in Q\}
$$
generated by them  is called the group generated by the automaton $A$. In particular, if the automaton $A$ is finite, then  $G(A)$ is a finitely generated subgroup of the group $FA(X)$.

In this paper we provide some naturally defined automaton representation of a free nonabelian group of rank 2. We realize this by a 2-state automaton $A=(X, Q, \varphi^A, \psi^A)$ over a changing alphabet, where  the changing alphabet $X=(X_1, X_2, \ldots)$, the set of states and the sequences $\varphi^A=(\varphi^A_1, \varphi^A_2,\ldots)$, $\psi^A=(\psi^A_1, \psi^A_2, \ldots)$ of transition and output functions are  defined as follows:
\begin{itemize}
\item[(a)] $X_i=\{1, 2, \ldots, r_i\}$, where $(r_1, r_2, \ldots)$ is an arbitrary nondecreasing and unbounded sequence of positive integers greater than~1,
\item[(b)] $Q=\{a, b\}$,
\item[(c)] $\varphi^A_i(q, x)=\left\{
\begin{array}{ll}
a, &x=1,\;q=b,\\
b,& x=1,\; q=a,\\
q, &x\neq1,
\end{array}
\right.$
\item[(d)]$
\psi^A_i(q, x)=\left\{
\begin{array}{ll}
\sigma_i(x), &q=a,\\
\tau_i(x), &q=b,
\end{array}
\right.
$
where $\tau_i$ is a transposition $(1,2)$ and $\sigma_i$ is an $r_i$-cycle $(1,2,\ldots, r_i)$  in the symmetric group of $X_i$.
\end{itemize}
The main result of this paper is the following
\begin{thm}\label{t1}
The group $G(A)=gp\{A_{1, a}, A_{1, b}\}$ is a free group of rank 2 generated freely by the automaton functions $A_{1, a}$ and $A_{1, b}$.
\end{thm}

For the proof, we show  how the concept of duality may be defined and studied for  automata over a changing alphabet.
At first, in Section~2, we provide  basic definitions concerning automata over a changing alphabet and automaton functions defined by such automata. In Section~3,  for a given automaton $(X, Q, \varphi, \psi)$ over a changing alphabet $X=(X_1, X_2, \ldots)$, we introduce the notion of a dual mapping $D_{i, x}\colon Q^*\to Q^*$ corresponding to  any letter $x\in X_i$ in the $i$-th ($i\geq 1$) transition of this automaton and we derive the mutual connection between the action of automaton functions and the action of dual mappings. In Section~4 we introduce the concept of stabilization of an automaton over a changing alphabet. In Section~5 we show that a certain finite and state-invertible automaton $B$ associated with the above defined automaton $A$ is stabilized, and further, in the last section, that the action of dual mappings corresponding to the automaton $B$ is especially convenient to study by using the concept of a pattern (presented also in \cite{11,12,13}). Finally, we show how the orbits of this action determine the properties of the group generated by the automaton $A$, and that the group $G(A)$ is freely generated by the transformations $A_{1, a}$ and $A_{1, b}$.

\section{Automata over a changing alphabet}

Let $\no=\{1,2,\ldots\}$ be the set of positive integers. The changing alphabet is an infinite sequence $X=(X_1, X_2,\ldots)$ of finite nonempty sets, which are called sets of letters. For every $i\in\no$ we consider the set
$$
X^*_{(i)}=\{x_1x_2\ldots x_m\colon x_j\in X_{i+j-1},\;m\in\no\}\cup\{\emptyset\}
$$
of all  words over the changing alphabet $X_{(i)}=(X_i, X_{i+1},\ldots)$, including the empty word $\emptyset$. The length of a word $w=x_1x_2\ldots x_m$ is denoted by $|w|$. If $w\in X^*_{(i)}$ and $v\in X^*_{(i+|w|)}$, then $wv$ denotes a concatenation of  $w$ and $v$. Obviously $wv\in X^*_{(i)}$.

\begin{defin}
An automaton $A$ over a changing alphabet $X=(X_1, X_2,\ldots)$ is a set $Q$ together with two infinite sequences
$\varphi=(\varphi_1, \varphi_2,\ldots)$, $\psi=(\psi_1, \psi_2, \ldots)$,
in which $\varphi_i\colon Q\times X_i\to Q$ and $\psi_i\colon Q\times X_i\to X_i$  are the so-called transition and output functions, respectively. The elements of $Q$ are called the internal states of the automaton $A$. The automaton is called finite if the set $Q$ of its internal states is finite.
\end{defin}

The  automaton $A=(X, Q, \varphi, \psi)$ may be interpreted as a machine, which being at a moment $i\in \no$  in a state $q\in Q$ and reading from the input tape a letter $x\in X_i$, goes to the state $\varphi_i(q, x)$, types on the output tape the letter $\psi_i(q, x)$,  moves both tapes to the next position and then proceeds further to the next moment $i+1$.  The above interpretation defines for any $i\in\no$ and any state $q\in Q$ a natural action $A_{i, q}$  on the set of words $X^*_{(i)}$ as follows
\begin{equation}\label{e1}
A_{i, q}(x_1x_2\ldots x_m)=\psi_i(q_1, x_1)\psi_{i+1}(q_2, x_2)\ldots \psi_{i+m-1}(q_m, x_m),
\end{equation}
where the states $q_j\in Q$ are defined recursively as follows: $q_1=q$ and $q_{j+1}=\varphi_{i+j-1}(q_j, x_j)$ for $1\leq j<m$.
We also let $A_{i, q}(\emptyset)=\emptyset$.

It is convenient to present an automaton $A=(X, Q, \varphi, \psi)$ as a labeled, directed, locally finite graph with the following set of vertices:
$$
\{(i, q)\colon i\in\no,\; q\in Q\}.
$$
Two vertices are connected with an arrow if and only if they are of the form $(i, q)$ and $(i+1, \varphi_i(q, x))$ for some $i\in\no$, $q\in Q$ and $x\in X_i$. This arrow is labeled by $x$, starts from the vertex $(i, q)$ and goes to the vertex $(i+1, \varphi_i(q, x))$. Each vertex $(i, q)$ is labeled by the corresponding state function
\begin{equation}\label{state-fun}
\sigma_{i,q}\colon X_i\to X_i,\;\;\;\sigma_{i,q}(x)=\psi_i(q,x).
\end{equation}
To make the graph of the automaton clear,  we will substitute a large number of arrows connecting two given vertices and having  the same direction for a one multi-arrow labeled by suitable letters and if the labelling of such a multi-arrow is obvious we will omit this labelling.

The image of a word $w=x_1x_2\ldots x_m$ under the action $A_{i, q}$ can be easily found using the graph of the automaton. One must find a directed path starting in a vertex $(i, q)$ with  consecutive labels $x_1, x_2, \ldots, x_m$. Such a path exists and is  unique. If $\pi_1, \pi_2, \ldots, \pi_m$ are the labels of consecutive vertices in this path, then the $j$-th letter ($1\leq j\leq m$) of the image $A_{i, q}(w)$ is equal to $\pi_j(x_j)$.

The transformation $A_{i, q}$ is called the automaton function corresponding to the state $q\in Q$ in the $i$-th transition of the automaton $A$. Every automaton function $A_{1, q}\colon X^*_{(1)}\to X^*_{(1)}$ corresponding to some state $q\in Q$ in the first transition of $A$ is called  the automaton function defined by $A$. The semigroup generated by the automaton functions  defined by $A$ we denote by $S(A)$:
$$
S(A)=sgp\{A_{1, q}\colon q\in Q\}.
$$

\begin{defin}
An automaton $A=(X, Q, \varphi, \psi)$ is called invertible if for every $i\in\no$ and every state $q\in Q$ the state function $\sigma_{i, q}\colon X_i\to X_i$ is invertible.
\end{defin}

If the automaton $A$ is invertible, then we define the inverse automaton of $A$ as the  automaton $I=(X, Q, \varphi', \psi')$ with the following sequences $\varphi'=(\varphi'_1, \varphi'_2,\ldots)$ and $\psi'=(\psi'_1, \psi'_2,\ldots)$ of transition and output functions:
$$
\varphi'_i(q, x)=\varphi_i(q,\sigma_{i, q}^{-1}(x)),\;\;\;\;\psi'_i(q, x)=\sigma_{i, q}^{-1}(x).
$$

\begin{prop}\label{prop123}
If the automaton $A$ is invertible and $I$ is the inverse automaton of $A$, then  the automaton functions $A_{i, q}\colon X_{(i)}^*\to X_{(i)}^*$  are all invertible and $A_{i, q}^{-1}=I_{i, q}$.
\end{prop}
\begin{proof}
Let $x_1x_2\ldots x_m\in X^*_{(i)}$. Let $y_1y_2\ldots y_m\in X^*_{(i)}$ be the image of $x_1x_2\ldots x_m$ under $A_{i, q}$ and let $z_1z_2\ldots z_m\in X^*_{(i)}$ be the image of $y_1y_2\ldots y_m$ under $I_{i, q}$. By definition of $A_{i, q}$ we have: $y_j=\psi_{i+j-1}(q_j, x_j)$ for $1\leq j\leq m$,
where the states $q_j$  are defined recursively as follows: $q_1=q$, $q_{j+1}=\varphi_{i+j-1}(q_j, x_j)$ ($1\leq j\leq m$).
By definition of $I_{i, q}$ we have: $z_j=\psi'_{i+j-1}(s_j, y_j)$ for $1\leq j\leq m$,
where the states $s_j$  are defined recursively as follows: $s_1=q$, $s_{j+1}=\varphi'_{i+j-1}(s_j, y_j)$ ($1\leq j\leq m$). Let us assume that $s_k=q_k$ for some $1\leq k\leq m$. Since $\sigma_{i+k-1, q_k}(x_k)=\psi_{i+k-1}(q_k, x_k)=y_k$, we have $\sigma_{i+k-1, q_k}^{-1}(y_k)=x_k$. Thus
\begin{eqnarray*}
s_{k+1}&=&\varphi'_{i+k-1}(s_k, y_k)=\varphi'_{i+k-1}(q_k, y_k)=\\
&=&\varphi_{i+k-1}(q_k,\sigma_{i+k-1,q_k}^{-1}(y_k))
=\varphi_{i+k-1}(q_k, x_k)=q_{k+1}.
\end{eqnarray*}
Since $s_1=q_1$, we obtain: $s_j=q_j$ for $1\leq j\leq m$. Thus for $1\leq j\leq m$ we have
$$
z_j=\psi'_{i+j-1}(s_j, y_j)=\psi'_{i+j-1}(q_j, y_j)
=\sigma^{-1}_{i+j-1, q_j}(y_j)=x_j.
$$
As a result we obtain $I_{i,q}A_{i,q}=Id_{X^*_{(i)}}$. By the above, we need only to show the invertibility of $A_{i, q}$. But since all considered maps preserves levels, which are finite sets, the condition $I_{i,q}A_{i,q}=Id_{X^*_{(i)}}$ implies that both maps are bijections on levels. In consequence these maps are invertible.
\end{proof}

Thus if  the automaton $A=(X, Q, \varphi, \psi)$ is invertible, then the automaton functions $A_{1, q}$ ($q\in Q$) generate a group, which we denote by $G(A)$ and call the group generated by $A$:
$$
G(A)=gp\{A_{1, q}\colon q\in Q\}.
$$

Let $A=(X, Q, \varphi, \psi)$ and $A'=(X, Q', \varphi', \psi')$ be two automata over the same changing alphabet $X=(X_1, X_2, \ldots)$. If we assume that the sets $Q$ and $Q'$ are disjoint, then we may define an automaton $A''=(X, Q\cup Q', \varphi'', \psi'')$ in which every transition function $\varphi''_i$ coincides with $\varphi_i$ when restricted to $Q\times X_i$  and with $\varphi'_i$ when restricted to $Q'\times X_i$; every output function $\psi''_i$ coincides with $\psi_i$ when restricted to $Q\times X_i$  and with $\psi_i'$ when restricted to $Q'\times X_i$. The automaton $A''$ is called the union of automata $A$ and $A'$ and is denoted by $A\cup A'$. If $Q\cap Q'\neq\emptyset$, then we may modify  the automata $A$ and $A'$ by renaming some of their states so that the sets of states in the modified automata are disjoint and then we may define the union $A\cup A'$ as the union of the modified automata.

\section{Dual mappings $D_{i, x}$ of an automaton}

Let $A=(X, Q, \varphi, \psi)$ be an automaton over the changing alphabet $X$. Let
$$
Q^*=\{q_1q_2\ldots q_n\colon q_j\in Q,\;n\in\no\}\cup\{\emptyset\}=\bigcup_{n\geq 0} Q^n
$$
be a free monoid over the set $Q$ of states. For any $i\in\no$ and any letter $x\in X_i$ the transition function $\varphi_i$ induces the mapping  $D_{i, x}\colon Q^*\to Q^*$ defined as follows:
\begin{equation}\label{e3}
D_{i, x}(q_1q_2\ldots q_n)=\varphi_i\left(q_1, x_1\right)\varphi_{i}\left(q_2, x_2\right)\ldots\varphi_{i}\left(q_n, x_n\right),
\end{equation}
where the letters $x_j\in X_{i}$ are defined recursively as follows: $x_1=x$ and $x_{j+1}=\psi_{i}\left(q_j, x_j\right)$ for $1\leq j<n$.
We also let $D_{i, x}(\emptyset)=\emptyset$. We call $D_{i, x}$ the dual mapping corresponding to the letter $x\in X_i$ in the $i$-th transition of the automaton $A$.

The dual mappings $D_{i, x}$ ($i\in\no$, $x\in X_i$) of the automaton $A$ can be defined in terms of a dual graph $\Gamma$ of $A$, which we define  as a labeled, directed graph with the set of vertices
$$
\{(i, x)\colon i\in\no,\;x\in X_i\}.
$$
Two vertices in $\Gamma$ are connected with an arrow if and only if they are of the form $(i, x)$ and $(i, \psi_i(q, x))$ for some $i\in\no$, $q\in Q$ and $x\in X_i$. This arrow is labeled by $q|\varphi_i(q, x)$, starts from the vertex $(i, x)$ and goes to the vertex $(i, \psi_i(q, x))$. The left entry $q$ of the label $q|\varphi_i(q, x)$ is referred to  as the input entry while the right entry $\varphi_i(q, x)$ is referred to  as the output entry. For a given $i\in\no$ the set of vertices $\{(i, x)\colon x\in X_i\}$ forms a component $\Gamma_i$ (not necessary connected) of the graph $\Gamma$ such that $\Gamma$ is a disjoint union of the components $\Gamma_i$. To compute $D_{i, x}(\xi)$ for some $\xi=q_1q_2\ldots q_n\in Q^*$ we  find a directed path in a component $\Gamma_i$ of $\Gamma$ which  starts in a vertex $(i, x)$ and such that the  consecutive input entries in the labelling of this path are $q_1, q_2, \ldots, q_n$. Such a path exists and is unique.  If $q_1', q_2', \ldots, q_n'$ are the consecutive output entries in this labelling, then $D_{i, x}(\xi)=q_1'q_2'\ldots q_n'$.

Let $i\in\no$ and $\xi=q_1q_2\ldots q_n\in Q^*$. We  define $A_{i, \xi}$ as the composition of automaton functions  corresponding to  consecutive states in $\xi$ in the $i$-th transition of the automaton $A$, that is $A_{i, \xi}=A_{i, q_n}\ldots A_{i, q_2} A_{i, q_1}$.
Also, we let $A_{i, \emptyset}=Id_{X^*_{(i)}}$. In particular,  we have $A_{i, \xi\eta}=A_{i,\eta}A_{i, \xi}$ for all $\xi, \eta\in Q^*$.
It is worth to see  that every element of the semigroup $S(A)$ is of the form $A_{1,\xi}$ for some $\xi\in Q^*$.
For any $w=x_1x_2\ldots x_m\in X_{(i)}^*$ we  also define $D_{i, w}$ as the composition of dual mappings of the automaton $A$ corresponding to  consecutive letters of $w$ as follows: $D_{i, w}=D_{i+m-1, x_m}\ldots D_{i+1, x_2}D_{i, x_1}$. Also, we let $D_{i, \emptyset}=Id_{Q^*}$. In particular $D_{i, wv}=D_{i+|w|, v}D_{i, w}$ for all $i\in\no$, $w\in X^*_{(i)}$ and $v\in X^*_{(i+|w|)}$.

\begin{prop}\label{p1}
For any $i\in\no$, $\xi,\eta\in Q^*$, $w\in X^*_{(i)}$ and $v\in X^*_{(i+|w|)}$ the following equalities hold
\begin{equation}\label{ee1}
A_{i, \xi}(wv)=A_{i, \xi}(w)A_{i+|w|, D_{i,w}(\xi)}(v).
\end{equation}
\begin{equation}\label{ee2}
D_{i, w}(\xi\eta)=D_{i, w}(\xi)D_{i, A_{i, \xi}(w)}(\eta),
\end{equation}
\end{prop}
\begin{proof}
If $\xi=\emptyset$ then
\begin{eqnarray*}
A_{i,\xi}(wv)&=&A_{i,\emptyset}(wv)=wv=A_{i, \emptyset}(w)A_{i+|w|,\emptyset}(v)=\\
&=&A_{i, \xi}(w)A_{i+|w|, D_{i, w}(\xi)}(v)
\end{eqnarray*}
and
$$
D_{i, w}(\xi\eta)=D_{i, w}(\eta)=\emptyset D_{i, A_{i, \emptyset}(w)}(\eta)=D_{i,w}(\xi)D_{i, A_{i, \xi}(w)}(\eta).
$$
In the same way we show that (\ref{ee1}) and (\ref{ee2}) hold in case $w=\emptyset$. Now, let us assume that $|\xi|=|w|=1$. Then $\xi=q$ and $w=x_0$ for some $q\in Q$ and  $x_0\in X_i$. By definition, for any $v=x_1\ldots x_m\in X^*_{(i+1)}$ we have
$$
A_{i, \xi}(wv)=A_{i, q}(x_0x_1\ldots x_m)
=\psi_i(s_0, x_0)\psi_{i+1}(s_1, x_1)\ldots\psi_{i+m}(s_m, x_m),
$$
where $s_0=q$ and $s_{j+1}=\varphi_{i+j}(s_j, x_j)$ for $0\leq j<m$. Since $D_{i, w}(\xi)=D_{i, x_0}(s_0)=\varphi_i(s_0, x_0)=s_1$, we have
$$
A_{i+|w|, D_{i, w}(\xi)}(v)=A_{i+1, s_1}(x_1\ldots x_m)
=\psi_{i+1}(s_1, x_1)\ldots\psi_{i+m}(s_m, x_m).
$$
Now, since $A_{i, \xi}(w)=A_{i, s_0}(x_0)=\psi_i(s_0, x_0)$, we obtain (\ref{ee1}) in this case.
Similarly, for any $\eta=q_1\ldots q_n\in Q^*$ we have
$$
D_{i, w}(\xi\eta)=D_{i, x_0}(q_0q_1\ldots q_n)
=\varphi_i(q_0, y_0)\varphi_i(q_1, y_1)\ldots\varphi_i(q_n, y_n),
$$
where $q_0=q$, $y_0=x_0$ and $y_{j+1}=\psi_i(q_j, y_j)$ for $0\leq j<n$. Since $A_{i,\xi}(w)=A_{i, q_0}(x_0)=\psi_i(q_0, x_0)=y_1$, we have
$$
D_{i, A_{i,\xi}(w)}(\eta)=D_{i, y_1}(q_1\ldots q_n)=\varphi_i(q_1, y_1)\ldots\varphi_i(q_n, y_n).
$$
Now, since $D_{i, w}(\xi)=D_{i, x_0}(q_0)=\varphi_i(q_0, x_0)$, we obtain (\ref{ee2}) in this case.

Further, we prove the proposition by induction on the sum of lengths of words $w$ and $\xi$. By the above, (\ref{ee1}) and (\ref{ee2}) hold for all  $\xi\in Q^*$ and $w\in X^*_{(i)}$ with $|\xi|+|w|\leq 2$. Now, let us take $n>2$ and let us assume that  (\ref{ee1}) and (\ref{ee2}) hold for all $\xi\in Q^*$ and $w\in X^*_{(i)}$ with $|\xi|+|w|<n$. Let $\xi\in Q^*$ and $w\in X^*_{(i)}$ be both nonempty and such that $|\xi|+|w|<n$ and let $x\in X_{|w|+i}$ and $q\in Q$. We show  the following two equalities:
$$
A_{i, \xi q}(wv)=A_{i, \xi q}(w)A_{i+|w|, D_{i, w}(\xi q)}(v)
$$
and
$$
D_{i, wx}(\xi\eta)=D_{i, wx}(\xi)D_{i, A_{i, \xi}(wx)}(\eta).
$$
By the inductive assumption for the first one we have:
\begin{eqnarray}\label{ee3}
A_{i, \xi q}(wv)=A_{i, q}A_{i,\xi}(wv)=A_{i, q}(A_{i, \xi}(w)A_{i+|w|, D_{i, w}(\xi)}(v)).
\end{eqnarray}
Since $|A_{i, \xi}(w)|=|w|$ and  $|q|+|A_{i, \xi}(w)|=1+|w|\leq |\xi|+|w|<n$, by the inductive assumption the right side of (\ref{ee3}) is equal to
\begin{eqnarray*}
A_{i, q}(A_{i, \xi}(w)A_{i+|w|, D_{i, w}(\xi)}(v))=\\
=A_{i,q}(A_{i, \xi}(w))A_{i+|w|, D_{i, A_{i, \xi}(w)}(q)}(A_{i+|w|, D_{i, w}(\xi)}(v))=\\
=A_{i, \xi q}(w)A_{i+|w|, D_{i,w}(\xi)D_{i, A_{i, \xi}(w)}(q)}(v)=\\
=A_{i, \xi q}(w)A_{i+|w|, D_{i, w}(\xi q)}(v).
\end{eqnarray*}
Similarly, for the second one we have:
\begin{eqnarray}\label{ee4}
D_{i, wx}(\xi\eta)=D_{i+|w|, x}D_{i, w}(\xi\eta)=D_{i+|w|,x}(D_{i, w}(\xi)D_{i, A_{i, \xi}(w)}(\eta)).
\end{eqnarray}
Since $|D_{i, w}(\xi)|=|\xi|$ and $|x|+|D_{i, w}(\xi)|=1+|\xi|\leq |w|+|\xi|<n$, we obtain by the inductive assumption for the right side of (\ref{ee4}):
\begin{eqnarray*}
D_{i+|w|,x}(D_{i, w}(\xi)   D_{i, A_{i, \xi}(w)}(\eta))=\\
=D_{i+|w|,x}(D_{i,w}(\xi))  D_{i+|w|, A_{i+|w|, D_{i,w}(\xi)}(x)}(D_{i, A_{i, \xi}(w)}(\eta))=\\
=D_{i, wx}(\xi)D_{i, A_{i, \xi}(w)A_{i+|w|, D_{i,w}(\xi)}(x)}(\eta)=\\
=D_{i, wx}(\xi)D_{i, A_{i, \xi}(wx)}(\eta).
\end{eqnarray*}
The inductive argument finishes the proof.
\end{proof}

We denote the semigroup  generated by dual mappings   of the automaton $A$ by $DS(A)$:
$$
DS(A)=sgp\{D_{i, x}\colon i\in\no,\;x\in X_i\}.
$$
Obviously, every mapping $D_{i, w}$ ($i\in\no$, $w\in X^*_{(i)}$) belongs to  $DS(A)$.

\begin{defin}
An automaton $A=(X, Q, \varphi, \psi)$ is called state-invertible if for any $i\in\no$ and any letter $x\in X_i$ the mapping $Q\ni q\mapsto \varphi_i(q, x)\in Q$ is invertible.
\end{defin}

\begin{prop}
If the automaton $A$ is state-invertible, then the mappings $D_{i, w}$ ($i\in \no$,  $w\in X^*_{(i)}$) are invertible.
\end{prop}
\begin{proof}
We need only to show that for any $i\in\no$ and any $x\in X_i$ the dual mapping $D_{i, x}$ is invertible. Since this mapping preserves the lengths of words, we must only to show that $D_{i, x}$ is one-to-one. So, let $q_1\ldots q_n\in Q^*$, $q_1'\ldots q_n'\in Q^*$ and let $1\leq j_0\leq n$ be the smallest number such that $q_{j_0}\neq q'_{j_0}$. By definition of a dual mapping we have $D_{i, x}(q_1\ldots q_n)=s_1\ldots s_n$,
where $s_j=\varphi_i(q_j, x_j)$ for $1\leq j\leq n$ and the letters $x_j\in X_i$ are defined recursively: $x_1=x$ and $x_{j+1}=\psi_i(q_j, x_j)$ for $1\leq j\leq n$. Also $D_{i, x}(q'_1\ldots q'_n)=s'_1\ldots s'_n$, where $s'_j=\varphi_i(q'_j, x'_j)$ for $1\leq j\leq n$ and the letters $x'_j\in X_i$ are defined recursively: $x'_1=x$ and $x'_{j+1}=\psi_i(q'_j, x'_j)$ for $1\leq j\leq n$. Since $x'_1=x_1=x$ and $q_j=q'_j$ for $1\leq j<j_0$, we obtain $x_j=x'_j$ for $1\leq j\leq j_0$. In consequence $s'_j=s_j$ for $1\leq j<j_0$. Since $q'_{j_0}\neq q_{j_0}$ and the mapping $q\mapsto \varphi_i(q, x_{j_0})$ is invertible, we obtain
$$
s_{j_0}'=\varphi_i(q'_{j_0}, x'_{j_0})=\varphi_i(q'_{j_0}, x_{j_0})\neq \varphi_i(q_{j_0}, x_{j_0})=s_{j_0}
$$
In consequence $D_{i, x}$ is one-to-one.
\end{proof}

In particular, if the automaton $A$ is state-invertible, then the dual mappings of $A$  are all invertible. We denote the group generated by them  by $DG(A)$:
$$
DG(A)=gp\{D_{i, x}\colon i\in\no,\;x\in X_i\}.
$$

\section{The stabilization of an automaton}
Let $A=(X, Q, \varphi, \psi)$ be an automaton over the changing alphabet $X$ and let $f, g\in DS(A)$.   For every $n\in\no$ we say that $f$ and $g$  are $n$-equivalent and we write $f\sim_n g$ if the restrictions of $f$ and $g$ to the set $Q^n$ coincide, that is $f(\xi)=g(\xi)$ for every $\xi\in Q^n$. Since the mappings from the semigroup $DS(A)$ preserve the lengths of words,  the relation $\sim_n$ is an equivalence relation on this semigroup, and if $A$ is invertible, then $\sim_n$ is an equivalence relation also on the group $DG(A)$.

Let $S$ and $T$ be two sets  of transformations of the monoid $Q^*$. We say that $S$ and $T$ are $n$-equivalent and we write $S\sim_n T$ if the following  condition holds: for any $f\in S$ there is $g\in T$ such that $f\sim_n g$ and for any $f\in T$ there is $g\in S$ such that $f\sim_n g$.

\begin{defin}
The numbers $i, i'\in\no$ are said to be $n$-equivalent if
$$
\{D_{i, x}\colon x\in X_i\}\sim_n \{D_{i', x}\colon x\in X_{i'}\}.
$$
If the numbers $i, i'\in\no$ are $n$-equivalent, then we write $i\sim_n i'$.
\end{defin}

\begin{defin}
The automaton $A$ is called stabilized if for any $n\in\no$ there is $\lambda_n\in\no$ such that $i\sim_n i'$ for any $i, i'\geq \lambda_n$. The sequence $(\lambda_1, \lambda_2, \ldots)$ is called the stabilizing sequence of $A$.
\end{defin}

\begin{prop}\label{prop1}
If the automaton $A$ is stabilized with the stabilizing sequence $(\lambda_1, \lambda_2, \ldots)$, then
$\{D_{i, w}\colon w\in X^*_{(i)}\}\sim_n \{D_{\lambda_n, w}\colon w\in X^*_{(\lambda_n)}\}$
for any $n\in\no$ and $i\geq \lambda_n$.
\end{prop}
\begin{proof}
Let $n\in\no$ and $i\geq \lambda_n$. By definition, if $w=x_1\ldots x_m\in X^*_{(i)}$, then $D_{i, w}=D_{i+m-1, x_m}\ldots D_{i, x_1}$. But for $1\leq j\leq m$ we have $i+j-1\sim_n \lambda_n+j-1$ and thus $D_{i+j-1, x_j}\sim_n D_{\lambda_n+j-1, y_j}$ for some $y_j\in X_{i'+j-1}$. Since $\sim_n$ is an equivalence relation on the semigroup $DS(A)$, we have: $D_{i, w}\sim_n D_{\lambda_n, v}$, where $v=y_1\ldots y_m\in X^*_{(\lambda_n)}$. Similarly, we obtain that for every $w\in X^*_{(\lambda_n)}$ there is $v\in X^*_{(i)}$ such that $D_{\lambda_n,w}\sim_n D_{i, v}$.
\end{proof}

\begin{prop}\label{prop2}
If the automaton $A$ is stabilized with the stabilizing sequence $(\lambda_1, \lambda_2, \ldots)$, then
\begin{equation}\label{e11}
sgp\{D_{i, w}\colon i\geq \lambda_n,\; w\in X^*_{(i)}\}\sim_n \{D_{\lambda_n, w}\colon w\in X^*_{(\lambda_n)}\}
\end{equation}
for any $n\in\no$. Moreover, if $A$ is finite and state-invertible then
\begin{equation}\label{e22}
gp\{D_{i, w}\colon i\geq \lambda_n,\; w\in X^*_{(i)}\}\sim_n \{D_{\lambda_n, w}\colon w\in X^*_{(\lambda_n)}\}
\end{equation}
for any $n\in\no$.
\end{prop}
\begin{proof}
Let $n\in\no$. By definition, every mapping $D_{i, w}$ with $i\geq \lambda_n$ and $w\in X^*_{(i)}$ is the composition of  dual mappings of the form $D_{i_j, x_j}$ for some $i_j\geq \lambda_n$ and $x_j\in X_{i_j}$. Thus every element $h$ from the semigroup on the left side of (\ref{e11}) is equal to the composition
$D_{i_1, x_1}D_{i_2, x_2}\ldots D_{i_m, x_m}$, where $i_j\geq \lambda_n$ and $x_j\in X_{i_j}$ for $1\leq j\leq m$. From the stabilization of  $A$ we obtain $i_j\sim_n \lambda_n+m-j$ for $1\leq j\leq m$. Thus $D_{i_j, x_j}\sim_n D_{\lambda_n+m-j, y_j}$ for some $y_j\in X_{\lambda_n+m-j}$. Since $\sim_n$ is an equivalence relation on the semigroup $DS(A)$, we have
$$
h\sim_n D_{\lambda_n+m-1, y_1}D_{\lambda_n+m-2, y_2}\ldots D_{\lambda_n, y_m}=D_{\lambda_n, w},
$$
where $w=y_my_{m-1}\ldots y_1\in X^*_{(\lambda_n)}$. On the other hand, every element of the form $D_{\lambda_n, w}$ belongs to the semigroup on the left side of (\ref{e11}). This implies the equivalence (\ref{e11}).

By the above, to prove (\ref{e22}) we need only to show that for any $i\geq \lambda_n$ and any letter $x\in X_i$ the inverse mapping $(D_{i, x})^{-1}$ is $n$-equivalent to the mapping $D_{\lambda_n, w}$ for some $w\in X^*_{(\lambda_n)}$. Since the dual mapping $D_{i, x}$ is invertible, the restriction of $D_{i, x}$ to the set $Q^n$ defines a permutation of this set. Since the automaton $A$ is finite, the set $Q^n$ is finite. In consequence, there is $p\in\no$ such that the restriction of the power  $D_{i, x}^p$ to the set $Q^n$ is an identity permutation of this set, that is $D_{i, x}^p\sim_n Id_{Q^*}$.  Since $\sim_n$ is an equivalence relation on the group $DG(A)$, we obtain $D_{i, x}^{-1}\sim_n D_{i, x}^{p-1}$. Thus from the equivalence (\ref{e11}) we obtain $D_{i, x}^{-1}\sim_n D_{i, x}^{p-1}\sim_n D_{\lambda_n, w}$    for some $w\in X^*_{(\lambda_n)}$. The claim follows.
\end{proof}

\section{The automaton generating a free group and the automaton $B$}\label{sec1}
Let $A=(X, Q, \varphi^A, \psi^A)$ be the  automaton over a changing alphabet defined by conditions (a)-(d) from the  introduction. Figure~1 shows the graph of the automaton $A$.
\begin{figure}[hbtp]
\centering
\includegraphics[width=9cm]{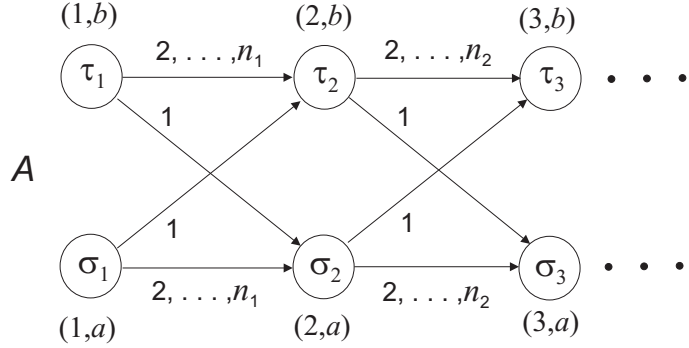}
\caption{The graph of the automaton $A$}
\label{fig11}
\end{figure}
For the state functions $\sigma_{i, q}\colon X_i\to X_i$ of $A$ we have: $\sigma_{i,a}=\sigma_i$ and $\sigma_{i, b}=\tau_i$. In particular the automaton  $A$ is invertible.
Let $I$ be the inverse automaton of $A$ and let $I_{-}$ be the automaton obtained from $I$ by renaming the states $a, b$ to $a^{-1}$ and $b^{-1}$ respectively. The symbols $a$, $b$, $a^{-1}$, $b^{-1}$ are regarded as elements of a free group with the basis $Q=\{a, b\}$.

Let us consider the union $B=A\cup I_{-}$
of automata $A$ and $I_{-}$. Then $B$ is an automaton over the same changing alphabet $X=(X_1, X_2, \ldots)$ and with the 4-element set $Q_{\pm}=\{a, b, a^{-1}, b^{-1}\}$ of internal states. Figure~2 shows the graph of the automaton $B$.
\begin{figure}[hbtp]
\centering
\includegraphics[width=9cm]{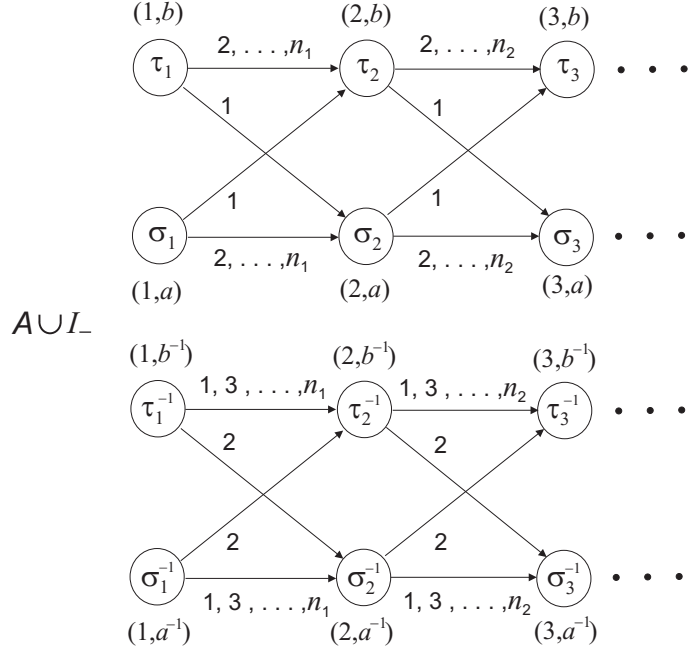}
\caption{The graph of the automaton $B=A\cup I_{-}$}
\label{fig9}
\end{figure}
By Proposition~\ref{prop123}, the automaton functions $B_{i, q}$ ($i\in\no$, $q\in Q_\pm$) satisfy:
\begin{equation}\label{e123}
B_{i, a}=A_{i,a},\;B_{i, b}=A_{i, b},\;B_{i, a^{-1}}=A_{i, a}^{-1},\;B_{i, b^{-1}}=A_{i, b}^{-1}.
\end{equation}
In consequence $B^{-1}_{i, \xi}=B_{i, \xi^{-1}}$ for every $\xi\in Q^*_{\pm}$. Also we have $G(A)=S(B)$.

The transition functions $\varphi_i\colon Q_\pm\times X_i\to Q_\pm$ of the automaton $B$ are described by formulas
\begin{equation}\label{e4}
\varphi_i(q, x)=\left\{
\begin{array}{ll}
q, & x\neq 1, 2,\\
a,& x=1,\;q=b,\\
b, &x=1, \;q=a,\\
a^{-1},& x=1,\;q=a^{-1},\\
b^{-1}, &x=1, \;q=b^{-1},\\
a, &x=2, \;q=a,\\
b, &x=2,\;q=b,\\
a^{-1}, &x=2, \;q=b^{-1},\\
b^{-1}, &x=2,\;q=a^{-1},\\
\end{array}
\right.
\end{equation}
For the output functions $\psi_i\colon Q_\pm\times X_i\to X_i$ of the automaton $B$ we have:
\begin{equation}\label{e5}
\psi_i(q, x)=\left\{
\begin{array}{ll}
\pi_{i,1}(x),&q=a,\\
\pi_{i, 2}(x),&q=a^{-1},\\
\pi_{i,3}(x),&q\in\{b, b^{-1}\},
\end{array}
\right.
\end{equation}
where $\pi_{i, 1}=\sigma_i$, $\pi_{i, 2}=\sigma_i^{-1}$ and $\pi_{i, 3}=\tau_i$.
From formula (\ref{e4}) we see that the automaton $B$ is state-invertible. Consequently, the dual mappings $D_{i, x}$ of this automaton and the mappings $D_{i, w}$ ($i\in\no$, $w\in X^*_{(i)}$) are all invertible. The component $\Gamma_i$ of the dual graph of the automaton $B$, which describes the action of  dual mappings $D_{i, x}$ ($x\in X_i$),  is presented in Figure~3.
\begin{figure}[hbtp]
\centering
\includegraphics[width=9cm]{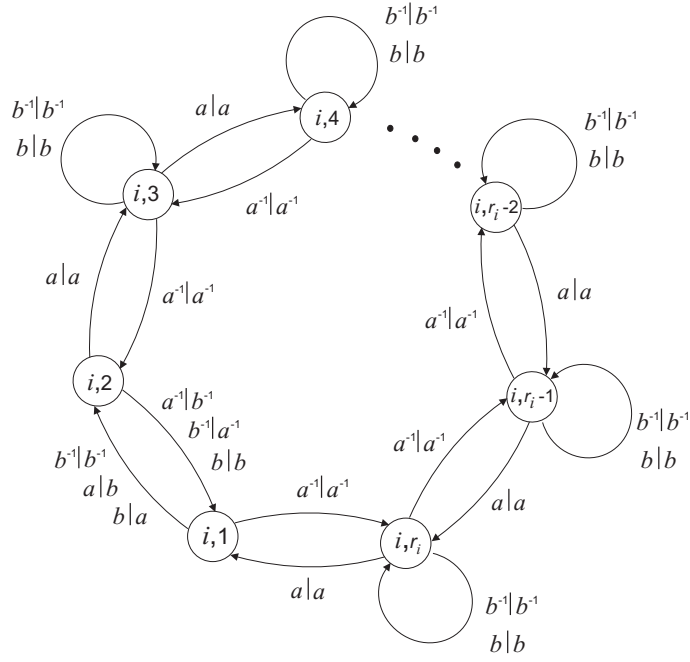}
\caption{The component $\Gamma_i$ of the dual graph of the automaton $B$}
\label{fig10}
\end{figure}

Let $\lambda_n$ ($n\in\no$)  be the smallest number such that $r_{\lambda_n}>2n$. Since the sequence $(r_1, r_2, \ldots)$ is unbounded, such a number exists. Obviously, the sequence $(\lambda_1, \lambda_2, \ldots)$ is nondecreasing.

\begin{prop}\label{prep}
The automaton $B$ is stabilized with the stabilizing sequence $(\lambda_1, \lambda_2, \ldots)$.
\end{prop}
\begin{proof}
Let $n\in\no$. Since the sequence $(r_1, r_2, \ldots)$ is nondecreasing, we have $r_i>2n$ for every $i\geq \lambda_n$. For any $i\geq \lambda_n$ let us define the subset $\widetilde{X}_i\subseteq X_i$ as follows
$$
\widetilde{X}_i=\{x\in X_i\colon x\leq n+1\}\cup\{x\in X_i\colon x>r_i-n+1\}.
$$
If $x\in X_i\setminus \widetilde{X}_i$, then $n+2\leq x\leq r_i-n+1$
and thus $|x-\pi_{i, \epsilon}(x)|\leq 1$ for any $\epsilon\in\{1,2,3\}$. Consequently, by easy inductive argument, for any sequence $(x_1, x_2, \ldots, x_n)$ in which $x_1\in X_i\setminus \widetilde{X}_i$, $x_{j+1}=\pi_{i, \epsilon_j}(x_j)$, $\epsilon_j\in\{1,2,3\}$ ($1\leq j\leq n$) we obtain: $n+3-j\leq x_j\leq r_i-n+j$.
In particular, for $1\leq j\leq n$ we have: $x_j\in X_i\setminus\{1,2\}$. Hence,  by formula~(\ref{e4}) we obtain  $\varphi_i(q, x_j)=q$ for  $q\in Q_\pm$ and $1\leq j\leq n$. Let $i\geq \lambda_n$ and $x\in X_i$. Then for every $\xi=q_1\ldots q_n\in Q_{\pm}^n$ we have
$
D_{i, x}(\xi)=\varphi_i(q_1, x_1)\ldots\varphi_i(q_n, x_n),
$
where $x_1=x$, $x_{j+1}=\pi_{i, \epsilon_j}(x_j)$, $\epsilon_j\in\{1,2,3\}$ ($1\leq j< n$). By the above, if $x\in X_i\setminus\widetilde{X}_i$, then $D_{i, x}(\xi)=\xi$ for every $\xi\in Q_{\pm}^n$. In particular, for every $i, i'\geq \lambda_n$, $x\in X_i\setminus\widetilde{X}_i$ and $y\in X_{i'}\setminus\widetilde{X}_{i'}$ we have:
$
D_{i, x}\sim_n D_{i', y}\sim_n Id_{Q^*_\pm}.
$
For  $i\geq \lambda_n$ and $\epsilon\in\{1,2\}$ let us define a letter $t_{i, \epsilon}\in X_i$ as follows
$$
t_{i,\epsilon}=\left\{
\begin{array}{ll}
r_i-n+1,&\epsilon =1,\\
n+2,& \epsilon=2.
\end{array}
\right.
$$
Then we easily verify that $
\widetilde{X}_i=\{\pi_{i, \epsilon}(t_{i, \epsilon}), \pi_{i, \epsilon}^2(t_{i, \epsilon}), \ldots, \pi_{i, \epsilon}^{2n}(t_{i, \epsilon})\}$.
Let us fix $i, i'\geq \lambda_n$. The mapping $\widetilde{X}_i\ni x\mapsto x'\in \widetilde{X}_{i'}$
defined as follows
$$
x'=\left\{
\begin{array}{ll}
x,&x\leq n+1,\\
r_{i'}-r_i+x, & x>r_i-n+1.
\end{array}
\right.
$$
is a bijection, for which we have:  if $x=r_i-n+j$ for some $1<j\leq n$ then $x'=r_{i'}-n+j$. This implies that if $x=\pi_{i, \epsilon}^k(t_{i, \epsilon})$ for some $\epsilon\in\{1,2\}$ and $1\leq k\leq 2n$,  then $x'=\pi_{i', \epsilon}^k(t_{i', \epsilon})$. From this observation for  $\epsilon\in \{1,2\}$ and $x\in \widetilde{X}_i$ we obtain the following two conditions: (i)  $\pi_{i, \epsilon}(x)\in \widetilde{X}_i$ if and only if $\pi_{i',\epsilon}(x')\in \widetilde{X}_{i'}$, (ii) if $\pi_{i, \epsilon}(x)\in \widetilde{X}_i$, then $(\pi_{i,\epsilon}(x))'=\pi_{i',\epsilon}(x')$. We easily check  (i) and (ii)  also for $\epsilon=3$. Beside that,  directly from formula~(\ref{e4}) we obtain: (iii) $\varphi_i(q, x)=\varphi_{i'}(q, x')$ for every $x\in\widetilde{X}_i$ and $q\in Q_{\pm}$. Now, let $x\in\widetilde{X}_i$. For every $\xi=q_1\ldots q_n\in Q_\pm^n$ there are $\epsilon_j\in\{1,2,3\}$, $1\leq j< n$ such that: $D_{i,x}(\xi)=\varphi_i(q_1, x_1)\ldots\varphi_i(q_n, x_n)$,
where $x_1=x$ and $x_{j+1}=\pi_{i, \epsilon_j}(x_j)$ for $1\leq j<n$. Similarly: $D_{i', x'}(\xi)=\varphi_{i'}(q_1, z_1)\ldots\varphi_{i'}(q_n, z_n)$, where $z_1=x'$ and $z_{j+1}=\pi_{i', \epsilon_j}(z_j)$ for $1\leq j<n$.
Let $1\leq j_0\leq n$ be the maximum number such that $x_{j}\in \widetilde{X}_i$ for $1\leq j\leq j_0$.
From (i)-(iii) we obtain: $z_j=x_j'$ and $\varphi_{i'}(q_j, z_j)=\varphi_i(q_j, x_j)$ for $1\leq j\leq  j_0$. Moreover, for $j_0<j\leq n$ we have: $x_j\in X_i\setminus\{1,2\}$ and $z_j\in X_{i'}\setminus\{1,2\}$. In particular $\varphi_i(q_j, x_j)=\varphi_{i'}(q_j, z_j)=q_j$ for $j_0<j\leq n$. As a result we obtain $D_{i, x}(\xi)=D_{i', x'}(\xi)$. Thus $D_{i, x}\sim_n D_{i', x'}$ for any $x\in \widetilde{X_i}$.   In consequence the sets $\{D_{i, x}\colon x\in X_i\}$ and $\{D_{i', x}\colon x\in X_{i'}\}$ are $n$-equivalent, that is $i\sim_n i'$ for any $i, i'\geq \lambda_n$.
\end{proof}

\section{The action of dual mappings}

Let us consider the set $\{\ast, \ast^{-1}\}$ of two symbols: $\ast$ and $\ast^{-1}$. Every word over $\{*, *^{-1}\}$ (that is every element of a free monoid generated by $\{*, *^{-1}\}$) is called a pattern. From every word $\xi\in Q^*_{\pm}$ we may obtain a pattern $V$ by substituting $\ast$ for each occurrence of letters $a, b$ and substituting $\ast^{-1}$ for each occurrence of letters $a^{-1}, b^{-1}$. Then we say that $V$ is the pattern of $\xi$ or that $\xi$ follows the pattern $V$. We say that a word $\xi=q_1q_2\ldots q_n\in Q_{\pm}^*$ is freely irreducible if none of its two-letter subwords $q_1q_2, q_2q_3, \ldots, q_{n-1}q_n$ coincides with one of the following words: $aa^{-1}$, $a^{-1}a$, $bb^{-1}$, $b^{-1}b$.

\begin{lem}\label{lem1}
For every $i\in\no$,  $w\in X^*_{(i)}$ and  $\xi\in Q^*_\pm$ the words $D_{i, w}(\xi)$ and $\xi$ follow the same pattern, and $D_{i, w}(\xi)$ is freely irreducible if and only if $\xi$ is freely irreducible.
\end{lem}
\begin{proof}
By Proposition~\ref{p1} for any dual mapping $D_{i, x}$  and any $\xi_1, \xi_2, \xi_3\in Q^*_\pm$, we have
$D_{i, x}(\xi_1\xi_2\xi_3)=D_{i, x}(\xi_1)D_{i, y}(\xi_2)D_{i, z}(\xi_3)$ for some $y, z\in X_i$. Since the automaton $B$ is a disjoint union of automata with sets of states $\{a, b\}$ and $\{a^{-1}, b^{-1}\}$ respectively, we obtain that
$D_{i, x}$ preserves both the set $\{a, b\}$ and the set $\{a^{-1}, b^{-1}\}$. In consequence, for every $\xi\in Q_{\pm}^*$ the words $\xi$ and $D_{i, x}(\xi)$ follow the same pattern. Directly by formula (\ref{e4}) or by the component $\Gamma_i$ of the dual graph of the automaton $B$ we see that  $D_{i, x}$ preserves the following sets: $\{a^{-1}b, b^{-1}a\}$, $\{ab^{-1}, ba^{-1}\}$, $\{aa^{-1}, bb^{-1}\}$ and $\{a^{-1}a, b^{-1}b\}$. As a result we obtain that $D_{i, x}(\xi)$ is freely irreducible if and only if $\xi$ is freely irreducible. Now, the claim follows from the fact that $D_{i, w}$ is a composition of dual mappings.
\end{proof}

Every nonempty pattern $V$ has a unique decomposition $V=V_IV_{II}$ satisfying the following two conditions:
\begin{itemize}
\item[(i)] if $V_{I}$ is nonempty, then the last letter of $V_I$ coincides with the first letter of $V_{II}$,
\item[(ii)] $V_{II}=(\ast\ast^{-1})^l\ast^{r}$ or $V_{II}=(\ast^{-1}\ast)^l\ast^{-r}$ for some $l\geq 0$, $r\in\{0,1\}$ and $(l, r)\neq (0,0)$. In particular $V_{II}$ is nonempty.
\end{itemize}
We call $V_I$ and $V_{II}$ respectively the first part and the second part of $V$.

Similarly, every nonempty word $\xi\in Q^*_{\pm}$ has a unique decomposition $\xi=\xi_{I}\xi_{II}$ into the first part $\xi_I$ and the second part $\xi_{II}$ described as follows: if $V$ is the pattern of $\xi$ and $V=V_{I}V_{II}$ is the corresponding decomposition of $V$, then $V_I$ is the pattern of $\xi_I$ and $V_{II}$ is the pattern of $\xi_{II}$. In particular, if $\xi$ is nonempty, then its second part $\xi_{II}$ is also nonempty.

For any $\xi\in Q^*_{\pm}$ we denote by $\widetilde{\xi}\in Q^*_\pm$  the word arising from $\xi$ by substitution of every $a$ for $b$, every $b$ for $a$, every $a^{-1}$ for $b^{-1}$ and every $b^{-1}$ for $a^{-1}$. Obviously $\xi$ and $\widetilde{\xi}$ follow the same pattern and $\xi$ is freely irreducible if and only if $\widetilde{\xi}$ is freely irreducible.

\begin{lem}\label{lem2}
If a word $\xi=\xi_I\xi_{II}\in Q^*_\pm$ is freely irreducible, then $D_{i, w}(\xi)=D_{i,w}(\xi_I)\xi_{II}$ or $D_{i,w}(\xi)=D_{i, w}(\xi_I)\widetilde{\xi_{II}}$ for any $i\in\no$ and any $w\in X^*_{(i)}$.
\end{lem}
\begin{proof}
The thesis is obvious in the case when $\xi$  is empty. So, let us assume that $\xi$ is nonempty. By Proposition~\ref{p1} we have $D_{i, w}(\xi)=D_{i,w}(\xi_I)D_{i, v}(\xi_{II})$ for some $v\in X^*_{(i)}$.
By Lemma~\ref{lem1} the words $D_{i, v}(\xi_{II})$ and $\xi_{II}$ are  freely irreducible and follow the same pattern. The words $\xi_{II}$ and $\widetilde{\xi_{II}}$ are distinct, freely irreducible and also follow the same pattern. But this pattern is equal to $(\ast\ast^{-1})^l\ast^r$ or to $(\ast^{-1}\ast)^l\ast^{-r}$ for some $l\geq 0$, $r\in\{0,1\}$ and for each of the above two patterns there are only two freely irreducible words which follow  this pattern: $(ab^{-1})^la^r$ and $(ba^{-1})^lb^r$ for $(\ast\ast^{-1})^l\ast^r$, and $(a^{-1}b)^la^{-r}$ and $(b^{-1}a)^lb^{-r}$ for  $(\ast^{-1}\ast)^l\ast^{-r}$.  In each case one of this word is equal to $\xi_{II}$ and the other is equal to $\widetilde{\xi_{II}}$. This implies $D_{i, v}(\xi_{II})=\widetilde{\xi_{II}}$ or $D_{i, v}(\xi_{II})=\xi_{II}$ in each case.
\end{proof}

\begin{prop}\label{p3}
For any nonempty pattern $V$ there is a freely irreducible word  $\xi\in Q^*_\pm$ which follows this pattern and such that $D_{i, z}(\xi)=\xi_{I}\widetilde{\xi_{II}}$ for some $z\in X_{i}$, where $i=\lambda_{|V|}$.
\end{prop}
\begin{proof}
Let $V=V_IV_{II}$. At first let us assume that $V_I\neq \emptyset$ and the last letter of $V_I$ is $\ast$. Then there is $k\geq 0$ such that
$$
V_I=\ast^{n_1}\ast^{-m_1}\ast^{n_2}\ast^{-m_2}\ldots \ast^{n_k}\ast^{-m_k}\ast^m,
$$
where  $n_1\geq 0$, $n_j>0$ for $2\leq j\leq k$, $m_j>0$ for $1\leq j\leq k$, $m>0$. For the second part of $V$ we have in this case: $V_{II}=(\ast\ast^{-1})^l\ast^r$
for some $l\geq 0$, $r\in\{0,1\}$ and $(l, r)\neq (0,0)$. Let $\xi=\xi_{I}\xi_{II}\in Q^*_\pm$ be a word with the first part
$$
\xi_I=a^{n_1}b^{-m_1}a^{n_2}b^{-m_2}\ldots a^{n_k}b^{-m_k}a^m
$$
and the second part $\xi_{II}=(ab^{-1})^la^r$. Then $\xi$ is freely irreducible and follows the pattern $V$. For the number $i=\lambda_{|V|}$ we have: $r_i>2|V|\geq N+m+1$, where $N=n_1+n_2+\ldots+n_k$. Also $D_{i, x}(q\eta)=\varphi_i(q, x)D_{i, \psi_i(q, x)}(\eta)$ for all $\eta\in Q^*_\pm$, $q\in Q_\pm$, $x\in X_i$.
Since $\psi_i(a, x)=x+1$ and $\varphi_i(a, x)=a$ for $1<x<r_i$, we obtain
$$
D_{i, z}(\xi)=a^{n_1}D_{i, z+n_1}(b^{-m_1}a^{n_2}b^{-m_2}\ldots a^{n_k}b^{-m_k}a^m\xi_{II}),
$$
where $z=r_i-N-m+1>2$.
Further, $\psi_i(b^{-1}, x)=x$ and $\varphi_i(b^{-1}, x)=b^{-1}$ for $2<x<r_i$. In consequence
$$
D_{i, z}(\xi)=a^{n_1}b^{-m_1}D_{i, z+n_1}(a^{n_2}b^{-m_2}\ldots a^{n_k}b^{-m_k}a^m\xi_{II}).
$$
Repeating this reasoning we obtain
$$
D_{i, z}(\xi)=a^{n_1}b^{-m_1}\ldots a^{n_k}b^{-m_k}a^{m-1} D_{i, z_0}(a\xi_{II}),
$$
where $z_0=z+N+m-1=r_i$.
Since $\psi_i(a, r_i)=1$ and $\varphi_i(a, r_i)=a$, we obtain
$$
D_{i, z_0}(a\xi_{II})=D_{i, r_i}(a\xi_{II})=aD_{i, 1}(\xi_{II}),
$$
and in consequence $D_{i, z}(\xi)=\xi_ID_{i, 1}(\xi_{II})$.
Further, since $\psi_i(a, 1)=2$, $\varphi_i(a, 1)=b$, $\psi_i(b^{-1}, 2)=1$ and $\varphi_i(b^{-1}, 2)=a^{-1}$,
we have for every $\eta\in Q^*_\pm$:
$$
D_{i, 1}(ab^{-1}\eta)=bD_{i, 2}(b^{-1}\eta)=ba^{-1}D_{i, 1}(\eta).
$$
Thus $D_{i, 1}((ab^{-1})^l\eta)=(ba^{-1})^lD_{i, 1}(\eta)$ and
$$
D_{i, 1}(\xi_{II})=D_{i, 1}((ab^{-1})^la^r)=(ba^{-1})^lD_{i,1}(a^r)=(ba^{-1})^lb^r=\widetilde{\xi_{II}}.
$$
Consequently $D_{i, z}(\xi)=\xi_{I}\widetilde{\xi_{II}}$. Now, let us assume that the last letter of $V_I$ is $\ast^{-1}$. Then there is $k\geq 0$ such that
$$
V_I=\ast^{-n_1}\ast^{m_1}\ast^{-n_2}\ast^{m_2}\ldots \ast^{-n_k}\ast^{m_k}\ast^{-m},
$$
where  $n_1\geq 0$, $n_j>0$ for $2\leq j\leq k$, $m_j>0$ for $1\leq j\leq k$, $m>0$. For the second part of $V$ we have in this case: $V_{II}=(\ast^{-1}\ast)^l\ast^{-r}$
for some $l\geq 0$, $r\in\{0,1\}$ and $(l, r)\neq (0,0)$. Let $\xi=\xi_{I}\xi_{II}\in Q^*_\pm$ be a word with the first part
$$
\xi_I=a^{-n_1}b^{m_1}a^{-n_2}b^{m_2}\ldots a^{-n_k}b^{m_k}a^{-m}
$$
and the second part $\xi_{II}=(a^{-1}b)^la^{-r}$. Then $\xi$ is freely irreducible and follows the pattern $V$. Similarly, as in the previous case, we  show that $D_{i, z}(\xi)=\xi_{I}\widetilde{\xi_{II}}$, where $i=\lambda_{|V|}$ and $z=N+m+2\in X_i$. Now, let us assume that $V_I=\emptyset$. Then $V=(\ast\ast^{-1})^l\ast^{r}$ or $V=(\ast^{-1}\ast)^l\ast^{-r}$ for some $l\geq 0$, $r\in\{0,1\}$ and $(l, r)\neq (0,0)$. In the first case we take $\xi=(ab^{-1})^la^{r}$ and in the second case we take $\xi= (a^{-1}b)^la^{-r}$. In particular $\xi_{I}=\emptyset$ and $\xi_{II}=\xi$. As before,  we verify for every $i\in\no$ that $D_{i, 1}(\xi)=\xi_{I}\widetilde{\xi_{II}}$ in the first case and $D_{i, 2}(\xi)=\xi_{I}\widetilde{\xi_{II}}$ in the second case.
\end{proof}

\begin{prop}\label{p4}
Let $\xi,\eta\in Q^*_\pm$ be two freely irreducible words following the same pattern. Then  $D_{\lambda_n, w}(\xi)=\eta$ for some $w\in X^*_{(\lambda_n)}$, where $n=|\xi|=|\eta|$.
\end{prop}
\begin{proof}
Let $V$ be the pattern of the words $\xi$ and $\eta$. We use the induction on the length $n$ of the pattern $V$. The thesis is obvious in case $n=0$. If $n=1$, then the thesis holds as for every $i\in\no$ we have $D_{i, 11}(a)=a$, $D_{i, 1}(a)=b$, $D_{i, 2}(b)=b$, $D_{i, 1}(b)=a$, $D_{i, 22}(a^{-1})=a^{-1}$, $D_{i, 2}(a^{-1})=b^{-1}$, $D_{i, 1}(b^{-1})=b^{-1}$ and $D_{i, 2}(b^{-1})=a^{-1}$.

Let $n=|V|\geq 2$ and let us assume that  the thesis holds for all nonempty patterns  with the length smaller than $n$.
Let $V=V_{I}V_{II}$ be the decomposition of $V$ into the first and the second part. Similarly, let $\xi=\xi_{I}\xi_{II}$ and $\eta=\eta_{I}\eta_{II}$ be decompositions of $\xi$ and $\eta$ in their first and second parts. By Proposition~\ref{p3} there is a freely irreducible word $\zeta=\zeta_{I}\zeta_{II}\in Q^*_\pm$ which follows the pattern $V$ and such that $D_{\lambda_n, z}(\zeta)=\zeta_{I}\widetilde{\zeta_{II}}$ for some  $z\in X_{\lambda_n}$.

The words $\xi_{II}$, $\eta_{II}$ and $\zeta_{II}$ are freely irreducible and follow the pattern $V_{II}$, which is equal to $(**^{-1})^l*^r$ or to $(*^{-1}*)^l*^{-r}$  for some $l\geq 0$ and $r\in\{0,1\}$. Also the words $\widetilde{\xi_{II}}$, $\widetilde{\eta_{II}}$ and $\widetilde{\zeta_{II}}$ are freely irreducible and follow the pattern $V_{II}$. As in the proof of Lemma~\ref{p4} we see that for each of the above two patterns there are only two freely irreducible words which follow this pattern. Since $\xi_{II}\neq \widetilde{\xi_{II}}$, $\eta_{II}\neq \widetilde{\eta_{II}}$ and $\zeta_{II}\neq \widetilde{\zeta_{II}}$, we obtain:
\begin{equation}\label{e2}
\left\{\xi_{II}, \widetilde{\xi_{II}}\right\}=\left\{\eta_{II}, \widetilde{\eta_{II}}\right\}=\left\{\zeta_{II}, \widetilde{\zeta_{II}}\right\}.
\end{equation}

The words $\xi_{I}$, $\eta_{I}$ and $\zeta_{I}$ are freely irreducible and follow the pattern $V_{I}$ with the length $k=|V_{I}|<|V|=n$. Thus, by the inductive assumption, we have $\zeta_I=D_{\lambda_k, v}(\xi_{I})=D_{\lambda_k, u}(\eta_{I})$ for some $v, u\in X^*_{(\lambda_k)}$. Since $\lambda_n\geq \lambda_k$, by Proposition~\ref{prop1} we have $D_{\lambda_k, v}\sim_k D_{\lambda_n, v'}$ and $D_{\lambda_k, u}\sim_k D_{\lambda_n, u'}$ for some $v', u'\in X^*_{(\lambda_n)}$. Then $\zeta_I=D_{\lambda_n, v'}(\xi_{I})=D_{\lambda_n, u'}(\eta_{I})$. Now, by Lemma~\ref{lem2} and by the equalities~(\ref{e2}) we see that each of the words $D_{\lambda_n, v'}(\xi)$ and $D_{\lambda_n, u'}(\eta)$ is equal to $\zeta_{I}\zeta_{II}=\zeta$ or to $\zeta_{I}\widetilde{\zeta_{II}}=D_{\lambda_n, z}(\zeta)$. If $D_{\lambda_n, v'}(\xi)=D_{\lambda_n, u'}(\eta)$, then $g(\xi)=\eta$, where $g=(D_{\lambda_n, u'})^{-1}D_{\lambda_n, v'}$. By Proposition~\ref{prop2} there is $w\in X^*_{(\lambda_n)}$ such that $D_{\lambda_n, w}(\xi)=g(\xi)=\eta$. If $D_{\lambda_n, v'}(\xi)\neq D_{\lambda_n, u'}(\eta)$, then one of the words $D_{\lambda_n, v'}(\xi)$, $D_{\lambda_n, u'}(\eta)$  is equal to $\zeta$ and the other is equal to $D_{\lambda_n, z}(\zeta)$. If $D_{\lambda_n, v'}(\xi)=\zeta$ and $D_{\lambda_n, u'}(\eta)=D_{\lambda_n, z}(\zeta)$, then $g(\xi)=\eta$, where $g=(D_{\lambda_n, u'})^{-1}D_{\lambda_n, z}D_{\lambda_n, v'}$. If $D_{\lambda_n, v'}(\xi)=D_{\lambda_n, z}(\zeta)$ and $D_{\lambda_n, u'}(\eta)=\zeta$, then $g(\xi)=\eta$, where $g=(D_{\lambda_n, u'})^{-1}(D_{\lambda_n, z})^{-1}D_{\lambda_n, v'}$. As in the previous case, we obtain by Proposition~\ref{prop2} that there is $w\in X^*_{(\lambda_n)}$ such that $D_{\lambda_n, w}(\xi)=g(\xi)=\eta$. The claim follows.
\end{proof}

\begin{prop}\label{prop21}
Let $\xi, \eta, \zeta\in Q^*_{\pm}$ be freely irreducible words following the same pattern. Then for every $k\in\no$ there are $w, v\in X^*_{(1)}$ such that $D_{1, w}(\xi)=\eta$, $D_{1, v}(\xi)=\zeta$ and $|w|=|v|\geq k$.
\end{prop}
\begin{proof}
Let $k\in\no$ and let us denote $n=|\xi|=|\eta|=|\zeta|$. Let $u\in X^*_{(1)}$ be any word of the length $|u|\geq \max\{k, \lambda_n\}$. Let us denote $D_{1, u}(\xi)=\xi'$. Then by Lemma~\ref{lem1} the word $\xi'$ is freely irreducible and follows the same pattern as $\xi$. Thus by proposition~\ref{p4} there is $u'\in X^*_{(\lambda_n)}$ such that $D_{\lambda_n, u'}(\xi')=\eta$. Since $1+|u|\geq \lambda_n$, by Proposition~\ref{prop1} we have $D_{1+|u|, u''}\sim_n D_{\lambda_n, u'}$ for some $u''\in X_{(1+|u|)}$. Then $D_{1+|u|, u''}(\xi')=\eta$. In consequence for the word $w=uu''\in X^*_{(1)}$ we have:
$$
D_{1, w}(\xi)=D_{1, uu''}(\xi)=D_{1+|u|, u''}D_{1, u}(\xi)=D_{1+|u|, u''}(\xi')=\eta.
$$
Also we have $|w|\geq |u|\geq \max\{k, \lambda_n\}$. Similarly, we show that there is a word $v\in X^*_{(1)}$ such that $|v|\geq \max\{k, \lambda_n\}$ and $D_{1, v}(\xi)=\zeta$. If $|w|=|v|$, then we obtain the thesis. So, let us assume $|v|<|w|$ and let us denote $l=|w|-|v|$. For $1\leq i\leq l$ let $x_i\in X_{i+|v|}$ be any letter such that $n+2\leq x_i\leq r_{i+|v|}-n+1$.
Such a letter exists as $r_{i+|v|}\geq r_{|v|}\geq r_{\lambda_n}>2n$. As in the proof of Proposition~\ref{prep} we show that the dual mappings $D_{|v|+i, x_i}$ ($1\leq i\leq l$) are $n$-equivalent to $Id_{Q^*_\pm}$. In particular, for the word $w'=x_1\ldots x_l\in X^*_{(1+|v|)}$ we obtain
$$
D_{1+|v|, w'}(\zeta)=D_{|v|+l, x_l}D_{|v|+l-1, x_{l-1}}\ldots D_{|v|+1, x_1}(\zeta)=\zeta.
$$
Consequently, for the word $vw'\in X^*_{(1)}$ we obtain
$$
D_{1, vw'}(\xi)=D_{1+|v|, w'}D_{1, v}(\xi)=D_{1+|v|, w'}(\zeta)=\zeta.
$$
Since $|vw'|=|v|+|w'|=|v|+l=|w|$, we obtain the thesis.
\end{proof}

\noindent
\begin{proof}[{\bf of Theorem~\ref{t1}}]
Let us assume that the group $G(A)$ generated by the automaton $A$ is not free. Then there is $s\in\no$ and the integers $n_1,\ldots, n_s$, $m_1, \ldots, m_s$ not all equal to 0 and such that $n_i\neq 0$ for $1<i\leq s$, $m_i\neq 0$ for $1\leq i<s$ and
\begin{equation}\label{e1234}
A_{1, a}^{n_1}A_{1, b}^{m_1}\ldots A_{1, a}^{n_s}A_{1, b}^{m_s}=Id_{X^*_{(1)}}.
\end{equation}
In particular, the word $\xi=a^{n_1}b^{m_1}\ldots a^{n_s}b^{m_s}\in Q^*_{\pm}$ is  freely irreducible and not empty. By the equalities (\ref{e123}) we obtain that the left side of (\ref{e1234}) is equal to $B_{1, \xi}$. Thus $B_{1, \xi}=Id_{X^*_{(1)}}$. Let $\xi_{II}$ be the second part of the word $\xi=\xi_I\xi_{II}$. Then there are $l\geq 0$ and $r\in\{0,1\}$ such that $(l,r)\neq (0,0)$ and
$$
\xi_{II}\in\{(ab^{-1})^la^r,\;(ba^{-1})^lb^r,\;(a^{-1}b)^la^{-r},\;(b^{-1}a)^lb^{-r}\}.
$$
By easy computation we obtain
\begin{equation}\label{e010}
\xi_{II}^{-1}\widetilde{\xi_{II}}\in\{\eta_1, \eta_1^{-1}, \eta_2, \eta_2^{-1}\},
\end{equation}
where $\eta_1=a^r(b^{-1}a)^{2l}b^{-r}$, $\eta_2=a^{-r}(ba^{-1})^{2l}b^r$. Let us take $k\in\no$  such that $r_k>2l+r+2$. By Proposition~\ref{prop21}  there are $w, v\in X^*_{(1)}$ with $|w|=|v|\geq k$ and such that $D_{1, w}(\xi)=\xi$ and $D_{1, v}(\xi)=\xi'$, where $\xi'=\xi_{I}\widetilde{\xi_{II}}$. Let us denote $i=1+|w|=1+|v|$. Then by Proposition~\ref{p1} we obtain
$B_{1, \xi}(wu)=B_{1, \xi}(w)B_{i, \xi}(u)$ and $B_{1, \xi}(vu)=B_{1, \xi}(v)B_{i, \xi'}(u)$ for all $u\in X^*_{(i)}$.
Since $B_{1, \xi}=Id_{X^*_{(1)}}$, we obtain $B_{i, \xi}=Id_{X^*_{(i)}}$ from the first of the above two equalities and $B_{i, \xi'}=Id_{X^*_{(i)}}$ from the second one. In consequence
$$
Id_{X^*_{(i)}}=B_{i, \xi'}(B_{i, \xi})^{-1}=B_{i,\xi'}B_{i, \xi^{-1}}=B_{i, \xi^{-1}\xi'}=B_{i,  \xi_{II}^{-1}\widetilde{\xi_{II}}}.
$$
Thus, by the membership (\ref{e010}) and the above equality we obtain: $B_{i, \eta_1}=Id_{X^*_{(i)}}$ or $B_{i, \eta_2}=Id_{X^*_{(i)}}$.  By formula~(\ref{e5}) the restrictions of $B_{i, a}$ and $B_{i, b}$ to the set of one-letter words coincide with $\sigma_i$ and $\tau_i$, respectively. Since we have:
\begin{eqnarray*}
B_{i, \eta_1}&=&B_{i, a^r(b^{-1}a)^{2l}b^{-r}}=B_{i, b}^{-r}(B_{i, a}B_{i, b}^{-1})^{2l}B_{i, a}^r,\\
B_{i, \eta_2}&=&B_{i, a^{-r}(ba^{-1})^{2l}b^r}=B_{i, b}^r(B_{i, a}^{-1}B_{i, b})^{2l}B_{i, a}^{-r},
\end{eqnarray*}
the restrictions of $B_{i, \eta_1}$ and $B_{i,\eta_2}$ to the set of one-letter words coincide with the permutations, respectively, $\pi_1$ and $\pi_2$ of the set $X_i$, where:
\begin{eqnarray*}
\pi_1&=&\tau_i^{-r}(\sigma_i\tau_i^{-1})^{2l}\sigma_i^r=\tau_i^r(\sigma_i\tau_i)^{2l}\sigma_i^r,\\
\pi_2&=&\tau_i^{r}(\sigma_i^{-1}\tau_i)^{2l}\sigma_i^{-r}.
\end{eqnarray*}
Since $i>k$, we have: $r_i\geq r_k>2l+r+2$. Beside that $\tau_i(x)=x$ for every $x\in X_i\setminus\{1, 2\}$. In consequence,  $\pi_1$ sends $3$ into $\pi_1(3)=\sigma_i^{2l+r}(3)\neq 3$. Similarly, $\pi_2$ sends $r_i$ into $\pi_2(r_i)=\sigma_i^{-2l-r}(r_i)\neq r_i$. In particular, neither $\pi_1$ nor $\pi_2$ is an identity permutation and we obtain the contradiction.
\end{proof}

{\bf Acknowledgements.} I am grateful to the referee  for pointing out the necessity for making changes in the introduction and  for detailed comments as well as for  remarkable insights concerning the  proofs of  Proposition~1, Lemma~1 and Proposition~8.

\end{document}